\documentclass[centering,11pt,reqno]{amsart}  

\usepackage[utf8]{inputenc}
\usepackage[T1]{fontenc}
\usepackage{amsmath,amsthm}
\usepackage{amsfonts,amssymb}
\usepackage[mathscr]{eucal}
\usepackage{url}
\usepackage{paralist}
\usepackage[colorlinks,cite color=blue,pagebackref=true,pdftex]{hyperref}
\usepackage[margin=2.5cm]{geometry}
\usepackage{tikz}
\usetikzlibrary{calc,shapes}
\usepackage{mathtools}
\usepackage{blkarray}
\usepackage[capitalize]{cleveref}
 \usepackage[foot]{amsaddr}
\usepackage{tikz-cd}

\newtheorem{theorem}{Theorem}[section]
\newtheorem*{lemma*}{Lemma}
\newtheorem*{prop*}{Proposition}
\newtheorem{lem}[theorem]{Lemma}

\newtheorem{op}[theorem]{Open Problem}

\theoremstyle{definition}

\renewcommand{\epsilon}{\varepsilon}

\definecolor{alexmcolor}{RGB}{9,6,250}
\definecolor{amandacolor}{RGB}{50,150,50}

\DeclareMathOperator*{\argmax}{arg\!\max}

\title{Exponential Lower Bounds for Many Pivot Rules for the Simplex Method}

\author[Black]{Alexander E. Black}
\address{Department of Mathematics, University of California, Davis, CA 95616}

\begin{document}

\maketitle

\begin{abstract}  The existence of a pivot rule for the simplex method that guarantees a strongly polynomial run-time is a longstanding, fundamental open problem in the theory of linear programming. The leading pivot rule in theory is the shadow pivot rule, which solves a linear program by projecting the feasible region onto a polygon. It has been shown to perform in expected strongly polynomial time on uniformly random instances and in smoothed analysis. In practice, the pivot rule of choice is the steepest edge rule, which normalizes the set of improving neighbors and then chooses a maximally improving normalized neighbor. Exponential lower bounds are known for both rules in worst-case analysis. However, for the shadow simplex method, all exponential examples were only proven for one choice of projection, and for the steepest edge rule, the lower bounds were only proven for the Euclidean norm. In this work, we construct linear programs for which any choice of projection for shadow rule variants will lead to an exponential run-time and exponential examples for any choice of norm for a steepest edge variant. 
\end{abstract}

\section{Introduction}

The existence of a strongly polynomial time algorithm for linear programming (LP) remains one of the most fundamental open problems in optimization as highlighted by Smale as his 9th problem for the $21$st century \cite{SmaleProblems}. The simplex method remains a candidate solution to this problem. However, despite many years of evidence of practical performance, it remains open whether one can even find a weakly polynomial worst-case guarantee for its run-time. In part, this is due to the many possible variants of the simplex method. Namely, geometrically, the simplex method solves a bounded, feasible linear program $\max(\mathbf{c}^{\intercal} \mathbf{x})$ subject to $A\mathbf{x} \leq \mathbf{b}$ by generating an initial vertex of the polytope defined by $\{\mathbf{x}: A\mathbf{x} \leq \mathbf{b}\}$ and then walking from vertex to vertex along edges such that at each step the linear objective function increases. Regardless of how the path is chosen, so long as the linear objective function increases at each step, the path will eventually reach an optimum. Such a path is called a \textbf{monotone path} and a way of choosing a monotone path is called a \textbf{pivot rule}. The length of the monotone path together with the complexity of making a single step determines the run-time of the simplex method. 

The key focus of the first part of this work will be \textbf{shadow pivot rules}, which choose a monotone path by projecting the polytope onto some two-dimensional polygon called the \textbf{shadow} implicitly and then lifting a path on the polygon back to a path on the polytope. Many of the strongest positive results for the simplex method rely on variants of the shadow simplex method. The first major breakthrough in this direction is due to Borgwardt in \cite{borgwardt, BorgwardtErratum} in the 1970's and 1980's in which he showed the shadow simplex method on a random linear program taken from any spherically symmetric probability distribution takes an expected $\Theta(n^{2}m^{\frac{1}{n-1}})$ steps for a linear program defined by $m$ facets in $n$ variables. Spielman and Teng introduced the smoothed analysis of algorithms in \cite{SpielmanTeng} with the key tool being the shadow simplex method in their analysis and all subsequent analyses \cite{vershynin, smoothedanalysis, SmoothUpperLower, ImprovedSmoothedAnalysis, KelnerSpielman}. In this case, each entry of $A$ and $\mathbf{b}$ for the linear program $\max(\mathbf{c}^{\intercal} \mathbf{x})$ such that $A \mathbf{x} \leq \mathbf{b}$ is perturbed by a Gaussian with standard deviation $\sigma$, and the best known upper bounds on the expected number of steps in this framework is $O(\sigma^{-3/2}n^{13/4}\log^{7/4}(m))$ from \cite{SmoothUpperLower}. Thus, in those two models of random linear programs, the simplex method performs in strongly polynomial time in expectation. There has furthermore been successful work on using random choices of shadows for a fixed linear program. Bounds in that case are in terms of $m$ and $n$ alongside a measure of the size of the data given by $\Delta$, the maximum absolute value of a subdeterminant of the constraint matrix, or other similar parameters \cite{DiscCurv, BrunschRoglin} and tend to be of the form $\text{poly}(m,n, \Delta)$. Deterministic versions of the shadow simplex method have also been proven to take an optimal, linear number of non-degenerate steps on $0/1$-polytopes \cite{01simplex} with more general bounds shown for lattice polytopes \cite{SODAPaper}. Given all of these positive results, we ask here whether there is a version of the shadow simplex method that guarantees a strongly polynomial time run-time.

It was shown first by Murty in \cite{Murty} and then by others \cite{01ProblemOrigin, DefProds, GoldfarbCube} that there are linear programs for which a shadow pivot rule could potentially take exponentially many steps. In fact, given any polygon with $\log(n)$ extension complexity, there is a corresponding polytope with an exponential sized shadow. However, the path chosen by the shadow pivot rule depends on the choice of shadow. In each known pathological example, it had been shown that there is one bad choice of shadow, but one could potentially choose a different shadow. For example, there exists a $0/1$-polytope with an exponential sized shadow \cite{01ProblemOrigin}, but in \cite{01simplex} they showed that one could always avoid that bad case by choosing a different shadow. In general, there are many possible paths that could be chosen depending on the shadow. For some linear programs such as on simplices, hyper-cubes, certain zonotopes and matroid independence polytopes, every path arises as the path chosen by some shadow \cite{FlagPoly, FiberPoly, ZonotopesCellStrings, cubepiles}. Even when not all paths show up as a shadow, there still may be many paths that do such as on the cross-polytope \cite{CrossPolyPaper}. Thus, exhibiting one shadow of exponential size does not imply that there is not some other polynomial sized shadow that could be chosen. We resolve that problem here in complete generality:

\begin{theorem}
\label{thm:shadowsimplex}
For each $n \in \mathbb{N}$, there exists an $n$-dimensional simple polytope $P$ with $3n$ facets, a vertex $\mathbf{v}$ of $P$, and $\mathbf{c} \in \mathbb{R}^{n}$ such that any shadow path on $P$ starting at $\mathbf{v}$ to optimize $\mathbf{c}^{\intercal} \mathbf{x}$ is of length at least $2^{n}$. Hence, there is no strongly polynomial time shadow simplex algorithm for linear programming regardless of how the shadow is chosen. 
\end{theorem}

Shadow paths appear in the combinatorics literature under the name \textbf{coherent monotone paths} \cite{FiberPoly}, and they play a central role in understanding the space of all monotone paths on a polytope. In particular, there is a CW-complex called the \textbf{Baues poset} built canonically as a topological space associated to the space of monotone paths for a fixed linear program \cite{CellStringsOnPolytopes}, and Billera and Sturmfels showed that there is a deformation retraction of this space onto a polytope whose vertices are precisely the coherent monotone paths \cite{FiberPoly}. One may then ask whether there always exists a polynomial length coherent monotone path between any pair of vertices on the polytope. Such a statement is different from asking for the shadow simplex method to perform polynomially many steps to solve a fixed linear program, since no linear program is specified. This question is then a variant of the \emph{polynomial Hirsch conjecture} \cite{HirschSolution, Hirsch}, which asks for a polynomial bound on the diameters of polytopes and is the central open problem in this area. Here we show the polynomial Hirsch conjecture is false if one is only allowed to follow coherent monotone paths:

\begin{theorem}
\label{thm:coherentpath}
For each $n \in \mathbb{N}$, there exists an $n$-dimensional simple polytope $P$ with $4n$ facets and a pair of vertices $\textbf{u}, \textbf{v}$ of $P$ such that any coherent path from $\mathbf{u}$ to $\mathbf{v}$ is of length at least $2^{n}$. In particular, every shadow of $P$ for which both $\textbf{u}$ and $\textbf{v}$ are projected to the boundary must have exponentially many vertices.
\end{theorem}

This theorem is surprising given that the best known lower bounds for the polynomial Hirsch conjecture are linear \cite{HirschSolution}, and the best known upper bounds on diameters of polytopes are subexponential but superpolynomial due originally to Kalai and Kleitman in \cite{KalaiKleitman} with more recent but still superpolynomial improvements \cite{ToddImproved, NoriImproved}. There is a long history of finding superpolynomial lower bounds for the performance of individual pivot rules for the simplex method starting with the 1972 work of Klee and Minty \cite{KleeMinty} with many follow up works \cite{CunninghamLowBound, Murty, GoldfarbCube, GoldfarbSit, jeroslowcube, BlandLowBound, RandomLowBounds, ZadehLowBound, ExponentialLowZadeh, DefProds, CrissCrossBad, ZadehWorstCase, ZadehNetworkFlow} covering the standard pivot rules used in theory and practice \cite{terlakysurvey}. Our perspective differs from those who came before us in that our goal is to prove lower bounds for large families of pivot rules. 

The steepest edge pivot rule given by maximizing $\frac{\mathbf{c}^{\intercal}(\mathbf{u} - \mathbf{v})}{\|\mathbf{u}-\mathbf{v}\|_{2}}$ is the pivot rule of choice in practice. Goldfarb and Sit constructed an exponential example \cite{GoldfarbSit} in 1979 via a modification of Klee and Minty's construction method. However, the choice of $2$-norm here is restrictive. We generalize their result to an infinite class of pivot rules:

\begin{theorem} \label{thm:anynorm}
For any norm $\eta: \mathbb{R}^{n} \to \mathbb{R}$ and $n \in \mathbb{N}$, there exists an $n$-dimensional combinatorial cube $C_{n}$ and $\mathbf{c} \in \mathbb{R}^{n}$ for which the pivot rule given by
\[\mathbf{v}^{i+1} = \argmax_{\mathbf{u} \in N_{\mathbf{c}}(\mathbf{v}^{i})} \frac{\mathbf{c}^{\intercal}(\mathbf{u} - \mathbf{v}^{i})}{\eta(\mathbf{u} -\mathbf{v}^{i})}, \]
where $N_{\mathbf{c}}(\mathbf{v}^{i})$ is the set of $\mathbf{c}$-improving neighbors of $\mathbf{v}^{i}$, follows a path through all vertices of $C_{n}$ to maximize $\mathbf{c}^{\intercal} \mathbf{x}$. 
\end{theorem}

This lower bound applies to all the $p$-steepest edge pivot rules studied in \cite{PSteepestEdge} and normalized weight pivot rules from \cite{pivotpolytopes} for which the weight is the linear objective function and the normalization does not depend on the the linear program. Our construction yields exponential examples for each norm, but not an example that fails for all norms simultaneously. However, we can construct such an example for many norms. Namely, we call a norm $\eta$ \textbf{regular} if $\eta(e_{i}) = 1$ for all $i \in [n]$. Then there is an example that fails for all regular norms simultaneously:

\begin{theorem}
\label{thm:allnormsatonce}
For each $n \in \mathbb{N}$, there exists an $n$-dimensional combinatorial cube $C_{n}$ and $\mathbf{c} \in \mathbb{R}^{n}$ such that every pivot rule of the form: 
\[\mathbf{v}^{i+1} = \argmax_{\mathbf{u} \in N_{\mathbf{c}}(\mathbf{v}^{i})} \frac{\mathbf{c}^{\intercal}(\mathbf{u} - \mathbf{v}^{i})}{\eta(\mathbf{u} -\mathbf{v}^{i})}, \]
where $N_{\mathbf{c}}(\mathbf{v}^{i})$ is the set of $\mathbf{c}$-improving neighbors of $\mathbf{v}^{i}$ and $\eta$ is a regular norm, follows a path through all vertices of $C_{n}$ to maximize $\mathbf{c}^{\intercal} \mathbf{x}$. In particular, the $p$-steepest edge simplex method will follow an exponential length path for all choices of $p$.
\end{theorem}

This is a mild assumption on our norm in that any norm satisfies it up to a precomposing with a positive diagonal matrix. Our strategy to prove all of our main theorems is to take a linear program with a single exponential length shadow path such as in \cite{GoldfarbCube} and modify it.

\section{Background}

We rely on the simplex method at the level of \cite{BertsimasBook} and then discuss pivot rules. The simplex method is generally defined for standard form linear programs $\max(\mathbf{c}^{\intercal} \mathbf{x})$ subject to $A \mathbf{x} = \mathbf{b}, \mathbf{x} \geq 0$ for $\mathbf{x} \in \mathbb{R}^{m}$ and $A$ an $n \times m$ matrix. A basic feasible solution $\mathbf{v} \in \mathbb{R}^{n}$ is any vector satisfying  $A \mathbf{v} = \mathbf{b}$, $\mathbf{v} \geq 0$, and $|\text{supp}(\mathbf{v})| \leq m$. For our purposes here, we may assume without loss of generality that the linear program is non-degenerate meaning that $|\text{supp}(\mathbf{v})| = m$ for all basic feasible solutions $\mathbf{v}$. Then the vertices of the polytope $P = \{\mathbf{x}: A \mathbf{x} = \mathbf{b}, \mathbf{x} \geq \mathbf{0}\}$ correspond to basic feasible solutions, and two basic feasible solutions $\mathbf{u}$ and $\mathbf{v}$ differ by an edge if $|\text{supp}(\mathbf{u}) \cap \text{supp}(\mathbf{v})| = m-1$.

We also assume the simplex method is given a starting basic feasible solution $\mathbf{v}$ with support $B$. We call $B$ the corresponding \textbf{basis}. Finding an initial basis is referred to as phase 1 of the simplex method. For our analysis, we assume the simplex method could start from any arbitrary initial basis. Hence, to prove lower bounds on the simplex method's performance, it suffices to find one initial basis from which the number of steps taken is exponential. A \textbf{pivot} is a change from the current basic feasible solution to another. Since we assumed the linear program is nondegenerate, pivoting corresponds to walking along an edge of the polytope. Each possible pivot is determined uniquely by a new choice of variable to add to the basis. Namely, let $A_{B}$ be the $n \times m$ matrix with zeroes in all coordinates other than those in the columns indexed by $B$ and equal to $A$ otherwise. Then, from the standard theory of the simplex method, for each $i \in [m] \setminus B$, there exists a unique $\lambda_{i} > 0$ such that $\mathbf{v} + \lambda_{i}(e_{i} - A_{B}^{-1} A_{i})$ is a basic feasible solution with support given by $i$ and all but one element of $B$. Hence, any choice of $i \in [n] \setminus B$ gives an edge, and all edges arise in that way. 

The simplex method solves the linear program by pivoting from basis to basis such that for each new $i$ added to a basis, $\mathbf{c}^{\intercal}(e_{i} - A_{B}^{-1} A_{i}) > 0$. If no such $i$ exists, the current basis is optimal. A \textbf{pivot rule} is a way of choosing $i$. For example, Dantzig's pivot rule says to choose $i$ maximizing $\mathbf{c}^{\intercal} (e_{i} - A_{B}^{-1} A_{i})$, the greatest improvement rule chooses the maximizer of $\mathbf{c}^{\intercal}\lambda_{i} (e_{i} - A_{B}^{-1} A_{i})$, and the steepest edge pivot rule chooses the maximizer of $\frac{\mathbf{c}^{\intercal} (e_{i} - A_{B}^{-1} A_{i})}{\|e_{i} - A_{B}^{-1} A_{i}\|_{2}}$.

Here we consider a generalized steepest edge rule, which chooses $i$ maximizing $\frac{\mathbf{c}^{\intercal} (e_{i} - A_{B}^{-1} A_{i})}{\eta(e_{i} - A_{B}^{-1} A_{i})}$ for some choice of norm $\eta$. Recall that a norm $\eta: \mathbb{R}^{n} \to \mathbb{R}$ is any function satisfying:
\begin{itemize}
    \item (Positivity) $\eta(\mathbf{x}) \geq 0$ with equality if and only if $\mathbf{x} = \mathbf{0}$.
    \item (Positive Homogeneity) $\eta(\lambda \mathbf{x}) = |\lambda| \eta(\mathbf{x})$ for all $\lambda \in \mathbb{R}$.
    \item (Triangle Inequality) $\eta(\mathbf{x} + \mathbf{y}) \leq \eta(\mathbf{x}) + \eta(\mathbf{y})$ for all $\mathbf{x}, \mathbf{y} \in \mathbb{R}^{n}$.
\end{itemize}

Shadow pivot rules require more explanation. Namely, there are many choices of a shadow rule. Each is given by a choice of $\mathbf{w}_{j} > 0$ for each $j \in [m] \setminus B_{0}$ for the initial basis $B_{0}$ to yield a vector $\mathbf{w} = \sum_{j \in [m] \setminus B_{0}} \mathbf{w}_{j} e_{j}$. Then each subsequent basis after a basis $B$ is chosen by maximizing $\frac{\mathbf{c}^{\intercal} (e_{i} - A_{B}^{-1} A_{i})}{\mathbf{w}^{\intercal}(e_{i} - A_{B}^{-1}A_{i})}$. Note that $\mathbf{w}$ is fixed in the first step and does not change later. Thus, computing the shadow pivot rule at each step may be done by the simplex method nearly as efficiently as computing Dantzig's original pivot rule. In particular, unlike greatest improvement, it does not require computing $\lambda_{i}$. 

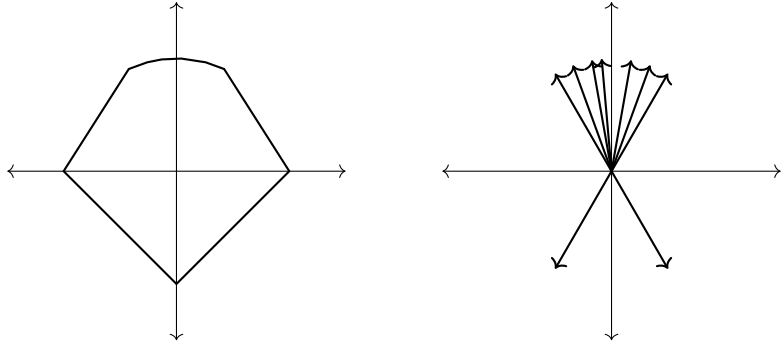
\begin{figure}
    \centering
    \[
\begin{tikzpicture}[scale = 1.5]
    \draw[<->] (-1.5,0) -- (1.5,0);
    \draw[<->] (0,-1.5) -- (0,1.5);
    \draw[thick] (0:1) -- (65:1) -- (75:1) -- (87.5:1) -- (97.5:1) -- (105:1) --(115:1) -- (180:1) -- (270:1) -- cycle;
\end{tikzpicture}
\hskip.5in
\begin{tikzpicture}[scale = 1.5]
    \draw[<->] (-1.5,0) -- (1.5,0);
    \draw[<->] (0,-1.5) -- (0,1.5);
    
    
    \draw[thick, ->] (0,0) -- (60:1);
    \draw[thick, ->] (0,0) -- (70:1);
    \draw[thick, ->] (0,0) -- (80:1);
    \draw[thick, ->] (0,0) -- (95:1);
    \draw[thick, ->] (0,0) -- (100:1);
    \draw[thick, ->] (0,0) -- (110:1);
    \draw[thick, ->] (0,0) -- (120:1);
    \draw[thick, ->] (0,0) -- (240:1);
    \draw[thick, ->] (0,0) -- (300:1);
\end{tikzpicture}\]
    \caption{Pictured is a polygon on the left with its corresponding normal fan on the right.}
    \label{fig:NormalFan}
\end{figure}

It remains to explain the connection to shadows. Consider the map $\pi: \mathbb{R}^{m} \to \mathbb{R}^{2}$ defined by the map $\pi(\mathbf{x}) = (\mathbf{w}^{\intercal} \mathbf{x}, \mathbf{c}^{\intercal} \mathbf{x})$. Then $\pi(P)$ is a polygon. Note that we call the optimum of a linear program for maximizing or minimizing $\mathbf{c}^{\intercal} \mathbf{x}$, a $\mathbf{c}$-maximizer or $\mathbf{c}$-minimizer respectively. We may assume without loss of generality that the $\mathbf{c}$-maximizer is unique. Furthermore, by construction, the initial vertex $\mathbf{u}$ is a $\mathbf{w}$-minimizer, since the support of $\mathbf{w}$ is $[n] \setminus B_{0}$. Hence, under the projection $\pi(\mathbf{u})$ minimizes the first coordinate. Similarly $\pi(\mathbf{v})$ maximizes the second coordinate, where $\mathbf{v}$ is the $\mathbf{c}$-maximizer. Then the shadow pivot rule maximizes $\frac{\mathbf{c}^{\intercal}(\mathbf{u} - \mathbf{v})}{\mathbf{w}^{\intercal}(\mathbf{u} -\mathbf{v})}$ over all $\mathbf{c}$-improving neighbors of $\mathbf{v}$. The quantity $\frac{\mathbf{c}^{\intercal}(\mathbf{u} - \mathbf{v})}{\mathbf{w}^{\intercal}(\mathbf{u} -\mathbf{v})}$ is precisely the slope of the edge in the projection $\pi$. An edge of maximal slope must stay on the boundary of the polygon. Hence, the pivot rule will follow the path on the boundary of the polygon from the point minimizing the first coordinate to the point maximizing the second. As a result, the total number of vertices of the shadow yields an upper bound on the length of the path.

Finally, a key tool in our proofs will be a third interpretation of shadow pivot rules on a dual view of polytopes. In fact, Gass and Saaty originally introduced the shadow simplex method in 1955 in this framework \cite{GassSaaty1955} as parametric linear programming. Namely, one can form a path from the starting vertex $\mathbf{u}$ to the optimum $\mathbf{v}$ by taking the path consisting of every vertex that maximizes $\lambda \mathbf{w} + (1-\lambda) \mathbf{c}$ for some $\lambda \in [0,1]$, where $\mathbf{w}$ and $\mathbf{c}$ are uniquely maximized at $\mathbf{u}$ and $\mathbf{v}$ respectively. This path is precisely the path followed by the shadow simplex method for $\pi(\mathbf{x}) = (-\mathbf{w}^{\intercal} \mathbf{x}, \mathbf{c}^{\intercal} \mathbf{x})$ and is best understood in terms of the normal fan of the polytope. Namely, given a polytope $P$, for each vertex $\mathbf{v}$, there is a cone given by $C_{\mathbf{v}} = \{\mathbf{z} \in \mathbb{R}^{n}: \mathbf{z}^{\intercal} \mathbf{v} \geq \mathbf{z}^{\intercal} \mathbf{x} \text{ for all }\mathbf{x} \in P\}$. That is $C_{\mathbf{v}}$ is the cone of all vectors maximized at $\mathbf{v}$. The extreme rays of $C_{\mathbf{v}}$ are precisely the facet defining inequalities tight at $\mathbf{v}$. The set of all such cones is called the \textbf{normal fan} of $P$. See Figure \ref{fig:NormalFan} for an example. Then, by definition, the shadow path for $\mathbf{w}$ consists of all vertices whose normal cones intersect the line segment from $\mathbf{w}$ to $\mathbf{c}$. See Figure \ref{fig:ShadowExamples} for an illustration of shadow paths from this perspective. Fundamentally for this work, the choice of $\mathbf{w}$ is not unique, and one may choose any $\mathbf{w}$ in the interior of $C_{\mathbf{u}}$ to choose a shadow path.

\begin{figure}
    \centering
\[
\begin{tikzpicture}[scale = 1.5]
    \draw[<->] (-1.5,0) -- (1.5,0);
    \draw[<->] (0,-1.5) -- (0,1.5);
    \draw[thick, violet] (0:1) -- (65:1) -- (75:1) -- (87.5:1) -- (97.5:1) -- (105:1) --(115:1) -- (180:1); 
    \draw[thick] (180:1) -- (270:1) -- (0:1);
    \draw (0:1) node[red, circle,fill, inner sep = 1.5pt] {};
    \draw (180:1) node[red, circle, fill, inner sep = 1.5pt] {};
    \draw (-1.2,.2) node[red] {$\mathbf{u}$};
    \draw (1.2,.2) node[red] {$\mathbf{v}$};
\end{tikzpicture}
\hskip.5in
\begin{tikzpicture}[scale = 1.5]
    \draw[<->] (-1.5,0) -- (1.5,0);
    \draw[<->] (0,-1.5) -- (0,1.5);
    
    
    \draw[blue, thick, dashed] (15:1) -- (150:1);
    \draw[red, thick, ->] (0,0) -- (15:1) node[right] {$\mathbf{c}$};
    \draw[thick, violet, ->] (0,0) -- (60:1);
    \draw[thick, violet, ->] (0,0) -- (70:1);
    \draw[thick, violet, ->] (0,0) -- (80:1);
    \draw[thick, violet, ->] (0,0) -- (95:1);
    \draw[thick, violet, ->] (0,0) -- (100:1);
    \draw[thick, violet, ->] (0,0) -- (110:1);
    \draw[thick, violet, ->] (0,0) -- (120:1);
    \draw[thick, ->] (0,0) -- (240:1);
    \draw[thick, ->] (0,0) -- (300:1);
    \draw[red, thick, ->] (0,0) -- (150:1) node[left] {$\mathbf{w}$};
\end{tikzpicture}\]

\[
\begin{tikzpicture}[scale = 1.5]
    \draw[<->] (-1.5,0) -- (1.5,0);
    \draw[<->] (0,-1.5) -- (0,1.5);
    \draw[thick] (0:1) -- (65:1) -- (75:1) -- (87.5:1) -- (97.5:1) -- (105:1) --(115:1) -- (180:1); 
    \draw[thick, violet] (180:1) -- (270:1) -- (0:1);
    \draw (0:1) node[red, circle,fill, inner sep = 1.5pt] {};
    \draw (-1.2,.2) node[red] {$\mathbf{u}$};
    \draw (1.2,.2) node[red] {$\mathbf{v}$};
    \draw (180:1) node[red, circle, fill, inner sep = 1.5pt] {};
\end{tikzpicture}
\hskip.5in
\begin{tikzpicture}[scale = 1.5]
    \draw[<->] (-1.5,0) -- (1.5,0);
    \draw[<->] (0,-1.5) -- (0,1.5);
    
    
    \draw[blue, thick, dashed] (210:1) -- (15:1);
    \draw[red, thick, ->] (0,0) -- (15:1) node[right] {$\mathbf{c}$};
    \draw[thick, ->] (0,0) -- (60:1);
    \draw[thick, ->] (0,0) -- (70:1);
    \draw[thick, ->] (0,0) -- (80:1);
    \draw[thick,  ->] (0,0) -- (95:1);
    \draw[thick,  ->] (0,0) -- (100:1);
    \draw[thick,  ->] (0,0) -- (110:1);
    \draw[thick, ->] (0,0) -- (120:1);
    \draw[thick, violet, ->] (0,0) -- (240:1);
    \draw[thick, violet, ->] (0,0) -- (300:1);
    \draw[red, thick, ->] (0,0) -- (210:1) node[left] {$\mathbf{w}$};
\end{tikzpicture}\]
    \caption{Pictured on the left is the shadow path on a polygon induced by $\mathbf{w}$ starting from the $\mathbf{w}$-minimal vertex $\mathbf{u}$ and walking to the $\mathbf{c}$-maximal vertex $\mathbf{v}$. On the right is the normal fan of the polygon and a depiction of the shadow path given by interpolating between $\mathbf{w}$ and $\mathbf{c}$. As one can see from this picture, different choices of $\mathbf{w}$ may yield different shadow paths.}
    \label{fig:ShadowExamples}
\end{figure}
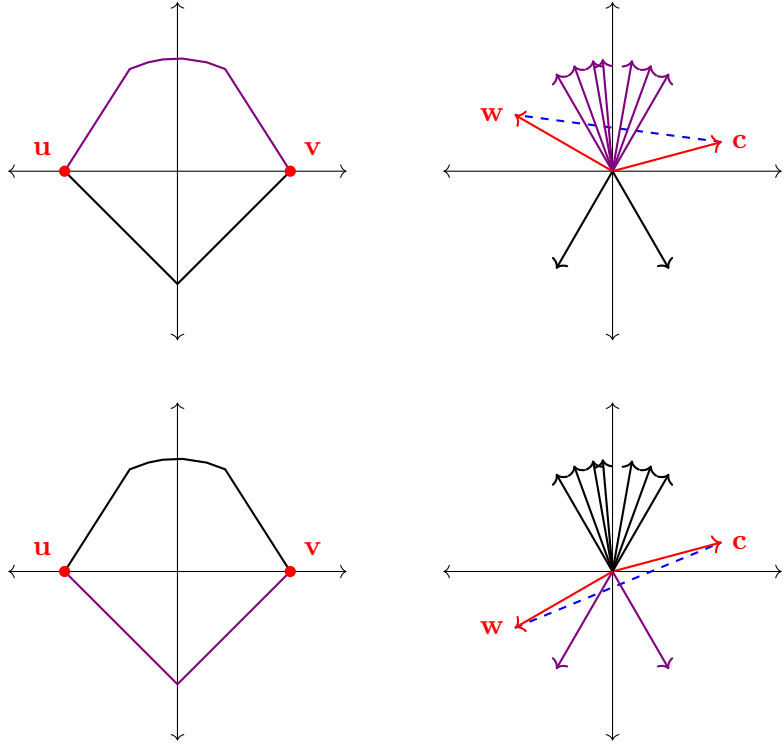
\section{Exponential Sized Shadows}

The main idea of the proof is as follows. Start with a polytope $P$ with an exponentially large shadow such as the Goldfarb cube \cite{GoldfarbCube}. Then, by the parametric definition of the shadow pivot rule, there exist $\mathbf{w}, \mathbf{c} \in \mathbb{R}^{n}$ such that $\mathbf{u}$ and $\mathbf{v}$ are a unique $\mathbf{w}$-maximizer and a unique $\mathbf{c}$-maximizer respectively. Using a polyhedral construction, we can construct a new polytope $P'$ with the same normal fan as $P$ except the normal cones $C_{\mathbf{u}}$ and $C_{\mathbf{v}}$ are subdivided. In particular, the resulting polytope has vertices $\mathbf{a}$ and $\mathbf{b}$ such that $\mathbf{w} \in C_{\mathbf{a}} \subset C_{\mathbf{u}}$ and $\mathbf{c} \in C_{\mathbf{b}} \subseteq C_{\mathbf{v}}$. We construct these new cones in such a way that the cones $C_{\mathbf{a}}$ and $C_{\mathbf{b}}$ are arbitrarily thin.  Then any shadow path from $\mathbf{a}$ to $\mathbf{b}$ must correspond to a line segment between points in these two thin cones.

By making the cones sufficiently thin, we may ensure that any such line segment must intersect every cone that the line segment from $\mathbf{w}$ to $\mathbf{c}$ intersects and must therefore always intersect exponentially many cones. Hence, any shadow path between the vertices corresponding to those two cones must be of exponential length. See Figure \ref{fig:Subdividing} for an illustration of the proof idea. In \cite{DiscCurv, Bonifas, BrunschRoglin}, each result relies on an assumption that all normal cones are sufficiently wide to ensure the existence of small shadows. Similarly, to find large shadows, we leverage that certain normal cones are thin.

The first key underlying observation from polyhedral theory is the normal fan interpretation of cutting off precisely one vertex and no other faces with the addition of a new inequality. Note that the extreme rays of the normal cone at a vertex are given precisely by the set of facet defining inequalities tight at the vertex as is noted in Chapter 7 of \cite{zieg}. Recall also that an $n$-dimensional polytope is \textbf{simple} if each vertex is contained in precisely $n$ facets and incident to precisely $n$ edges. Here we also identify a fan by its set of maximal dimensional cones $\mathcal{N}$ and a cone by its extreme rays.

\begin{lem}
\label{lem:subdiv}
Let $P$ be an $n$-dimensional simple polytope with normal fan $\mathcal{N}$, and choose a cone $C \in \mathcal{N}$ generated by $\mathbf{x}_{1}, \mathbf{x}_{2}, \dots, \mathbf{x}_{n}$. Let $\mathbf{w}$ be on the interior of $C$. Then there is a polytope $P'$ with normal fan $\mathcal{N}'$ such that $\mathcal{N}' = \mathcal{N} \setminus \{C\} \cup \{C_{1}, \dots, C_{m}\}$, where $C_{i}$ is the cone generated by $\mathbf{x}_{1}, \mathbf{x}_{2}, \dots, \mathbf{x}_{i-1}, \mathbf{w}, \mathbf{x}_{i+1}, \dots, \mathbf{x}_{n}$. In particular $P'$ has precisely one more facet than $P$.
\end{lem}

\begin{proof}
Since $\mathbf{w}$ is on the interior of $C$, $\mathbf{w}$ is uniquely maximized at a vertex $\mathbf{v}$ of $P$ corresponding to $C$. Let $P'= P \cap \{\mathbf{x}: \mathbf{w}^{\intercal}\mathbf{x} \leq \mathbf{w}^{\intercal} \mathbf{v} - \varepsilon\}$ for $\varepsilon > 0$ so small that $\mathbf{c}^{\intercal} \mathbf{v} > \mathbf{v}^{\intercal} \mathbf{u} + \varepsilon$ for all neighbors $\mathbf{u}$ of $\mathbf{v}$, which exists by the uniqueness of $\mathbf{v}$ as a $\mathbf{c}$-optimum. Then the hyper-plane defined by $\mathbf{w}$ is supporting and intersects each edge incident to $\mathbf{v}$ on its interior, and those are exactly the new vertices created by intersecting $P$ with that half-space. All other vertices and normal cones remain the same. The facet directions tight at each new vertex are given by $\mathbf{w}$ and the set of facets tight at the corresponding edge incident to $\mathbf{v}$. Hence, the set of cones for each of those vertices is precisely $C_{1}, C_{2}, \dots, C_{n}$. All other cones remain unchanged. 
\end{proof}

This lemma tells us that cutting off a single vertex with a new inequality slicing through the interiors of all of the edges incident to that vertex corresponds to a combinatorial barycentric subdivision of the normal cone. It also preserves nondegeneracy, since the resulting polytope is simple and being simple is equivalent to being non-degenerate. We may push this even further to find that given a simple polytope with a given normal fan, we can construct another polytope with the same normal fan except one cone is subdivided such that the generators of extreme rays of one of the cones in the subdivision are contained in an $\varepsilon$ ball of our choosing. 

\begin{lem}
\label{lem:improvedsubdiv}
Let $P$ be an $n$-dimensional simple polytope with normal fan $\mathcal{N}$ and $m$ facets, and let $C \in \mathcal{N}$. Let $D \subseteq C$ be an open ball of radius $\varepsilon > 0$. Then there exists an $n$-dimensional polytope $P'$ with normal fan $\mathcal{N}'$ with $m + n$ facets such that $\mathcal{N}' \supseteq \mathcal{N} \setminus \{C\}$ and there exists a cone $E \in \mathcal{N}'$ such that the extreme rays of $E$ all have generators in $D$.   
\end{lem}

\begin{proof}
We apply Lemma \ref{lem:subdiv} repeatedly. First choose any vector $\mathbf{w}_{1}$ in $D$. Let $C$ be the cone generated by $\mathbf{x}_{1}, \mathbf{x}_{2}, \dots, \mathbf{x}_{n}$. Then by Lemma \ref{lem:subdiv}, there is a polytope $P_{1}$ with the same normal fan as $P$ except the cone $C$ is replaced by $C_{1}, C_{2}, \dots, C_{n}$, where $C_{i}$ has extreme rays $\mathbf{x}_{1}, \mathbf{x}_{2}, \dots, \mathbf{x}_{i-1}, \mathbf{w}_{1},$ $ \mathbf{x}_{i+1}, \dots, \mathbf{x}_{n}$. In particular, the resulting polytope has a normal cone of the form $\mathbf{w}_{1}, \mathbf{x}_{2}, \dots, \mathbf{x}_{n}$. We then build the desired cone using Lemma \ref{lem:subdiv} and induction. Namely, let $k < n$, and suppose that there is a cone with extreme rays $\mathbf{w}_{1}, \mathbf{w}_{2}, \dots, \mathbf{w}_{k}, \mathbf{x}_{k+1}, \dots, \mathbf{x}_{n}$ in a polytope $P_{k}$ with $k$ more facets than $P$ with additional facet normals of the form $\mathbf{w}_{1}, \dots, \mathbf{w}_{k} \in D$. Suppose further that the normal fan of $P_{k}$ contains $\mathcal{N} \setminus \{C\}$. Then define the vector:
\[\mathbf{w}_{k+1} = \sum_{i=1}^{k} \frac{\mathbf{w}_{i}}{k} + \delta \sum_{j=k+1}^{n} \mathbf{x}_{j}\] 
for $\delta > 0$ so small that $\mathbf{w}_{k+1} \in D$. Such a $\delta$ exists, since $\mathbf{w}_{1}, \dots, \mathbf{w}_{k} \in D$ and $D$ is open and convex. Then $\mathbf{w}_{k+1}$ is on the interior of the cone generated by $\mathbf{w}_{1}, \dots, \mathbf{w}_{k}, \mathbf{x}_{k+1}, \dots, \mathbf{x}_{n}$, since it has positive coefficients for each of the extreme rays. By Lemma \ref{lem:subdiv}, we may then construct a new polytope $P_{k+1}$ with $k+1$ more facets than $P$ such that it contains a cone with $k+1$ extreme rays with generators in $D$ and such that its normal fan contains $\mathcal{N} \setminus \{C\}$. By induction, after doing this $n$ times, we arrive at a polytope $P'$ satisfying the desired conditions.

\end{proof}

To prove both Theorems \ref{thm:shadowsimplex} and \ref{thm:coherentpath}, we will prove a more general statement saying that whenever a shadow path of a given length exists, we may construct a new simple polytope for which every shadow path between a fixed pair of vertices must be of at least that length while only adding a linear number of facets.

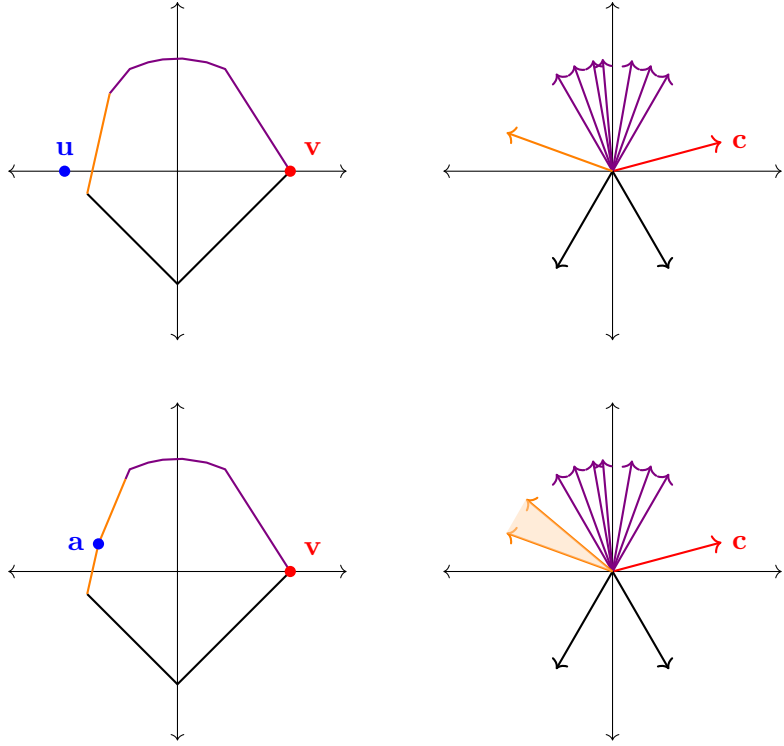
\begin{figure}
    \centering
    \[
\begin{tikzpicture}[scale = 1.5]
    \draw[<->] (-1.5,0) -- (1.5,0);
    \draw[<->] (0,-1.5) -- (0,1.5);
    \draw[thick, violet] (0:1) -- (65:1) -- (75:1) -- (87.5:1) -- (97.5:1) -- (105:1) --(115:1)-- (-.6, .69); 
    \draw[thick] (-.8,-.2) -- (270:1) -- (0:1);
    \draw[thick, orange] (-.8,-.2) -- (-.6, .69);
    \draw (0:1) node[red, circle,fill, inner sep = 1.5pt] {};
    \draw (1.2,.2) node[red] {$\mathbf{v}$};
    \draw (-1,0) node[blue, circle,fill, inner sep = 1.5pt] {};
    \draw (-1,.2) node[blue] {$\mathbf{u}$};
    
\end{tikzpicture}
\hskip.5in
\begin{tikzpicture}[scale = 1.5]
    \draw[<->] (-1.5,0) -- (1.5,0);
    \draw[<->] (0,-1.5) -- (0,1.5);
    
    
    \draw[red, thick, ->] (0,0) -- (15:1) node[right] {$\mathbf{c}$};
    \draw[thick, violet, ->] (0,0) -- (60:1);
    \draw[thick, violet, ->] (0,0) -- (70:1);
    \draw[thick, violet, ->] (0,0) -- (80:1);
    \draw[thick, violet, ->] (0,0) -- (95:1);
    \draw[thick, violet, ->] (0,0) -- (100:1);
    \draw[thick, violet, ->] (0,0) -- (110:1);
    \draw[thick, violet, ->] (0,0) -- (120:1);
    \draw[thick, ->] (0,0) -- (240:1);
    \draw[thick, ->] (0,0) -- (300:1);
    \draw[orange, thick, ->] (0,0) -- (160:1);
\end{tikzpicture}\]

\[
\begin{tikzpicture}[scale = 1.5]
    \draw[<->] (-1.5,0) -- (1.5,0);
    \draw[<->] (0,-1.5) -- (0,1.5);
    \draw[thick, violet] (0:1) -- (65:1) -- (75:1) -- (87.5:1) -- (97.5:1) -- (105:1) --(115:1)-- (-.46,.82); 
    \draw[thick] (-.8,-.2) -- (270:1) -- (0:1);
    \draw[thick, orange] (-.8,-.2) -- (-.7, .245) -- (-.46,.82);
    \draw (0:1) node[red, circle,fill, inner sep = 1.5pt] {};
    \draw (1.2,.2) node[red] {$\mathbf{v}$};
    \draw (-.7, .245) node[blue, circle, fill, inner sep = 1.5pt] {};
    \draw (-.9, .245) node[blue] {$\mathbf{a}$};
\end{tikzpicture}
\hskip.5in
\begin{tikzpicture}[scale = 1.5]
    \draw[<->] (-1.5,0) -- (1.5,0);
    \draw[<->] (0,-1.5) -- (0,1.5);
    
    
    \draw[red, thick, ->] (0,0) -- (15:1) node[right] {$\mathbf{c}$};
    \draw[thick, violet, ->] (0,0) -- (60:1);
    \draw[thick, violet, ->] (0,0) -- (70:1);
    \draw[thick, violet, ->] (0,0) -- (80:1);
    \draw[thick, violet, ->] (0,0) -- (95:1);
    \draw[thick, violet, ->] (0,0) -- (100:1);
    \draw[thick, violet, ->] (0,0) -- (110:1);
    \draw[thick, violet, ->] (0,0) -- (120:1);
    \draw[thick, ->] (0,0) -- (240:1);
    \draw[thick, ->] (0,0) -- (300:1);
    \draw[orange, thick, ->] (0,0) -- (140:1);
    \draw[orange, thick, ->] (0,0) -- (160:1);
    \fill[orange!30, opacity=0.5] (0,0) -- (140:1) -- (160:1) -- cycle;
\end{tikzpicture}\]
    \caption{Pictured at the top is an example of subdividing the normal fan of the polygon from Figure \ref{fig:NormalFan} by adding a single new inequality as in Lemma \ref{lem:subdiv}. Note that, since this inequality does not cut off any edges completely and only removes a vertex, the resulting normal fan is a refinement of the original normal fan. Pictured at the bottom is the result of adding a second inequality to create a new vertex $\mathbf{a}$ whose normal cone is properly contained in the original normal cone of the cutoff vertex $\mathbf{u}$ as constructed by Lemma \ref{lem:improvedsubdiv}. Note that any line segment from a point in the normal cone of the blue vertex to $\mathbf{c}$ must follow the longer path on the top of the polygon. This is precisely the idea of the proof of Theorem \ref{thm:manyfromone}.  }
    \label{fig:Subdividing}
\end{figure}

\begin{theorem} 
\label{thm:manyfromone}
Let $P$ be a simple polytope with $m$ facets in $n$ dimensions, and suppose that $P$ has a shadow path of length $\alpha$ from a vertex $\mathbf{u}$ to a vertex $\mathbf{v}$. Then there exists a simple polytope $Q$ with $m+2n$ facets in $n$ dimensions with vertices $\mathbf{a}$ and $\mathbf{b}$ such that any shadow path from $\mathbf{a}$ to $\mathbf{b}$ is of length at least $\alpha$.
\end{theorem}

\begin{proof}
 Let $\mathbf{w}$ and $\mathbf{c}$ be linear objective functions optimized uniquely at $\mathbf{u}$ and $\mathbf{v}$ respectively such that the segment from $\mathbf{w}$ and $\mathbf{c}$ intersected with the normal fan of $P$ determines the path. Consider the cones intersected by the segment $\{\lambda \mathbf{w} + (1-\lambda)\mathbf{c}: \lambda \in [0,1]\}$ in order from $\mathbf{u}$ to $\mathbf{w}$ given by $C_{0}, C_{1}, \dots, C_{\alpha}$, where $\alpha$ is the length of the path. Each $C_{i}$ is intersected on its interior by the line segment, so for each $i$, we may choose a vector $\mathbf{c}_{i}$ on the segment that lies on the interior of the cone $C_{i}$ with $\mathbf{w} = \mathbf{c}_{0}$ and $\mathbf{c} = \mathbf{c}_{\alpha}$. Then, for each $\mathbf{c}_{i}$ there exists some $\varepsilon_{i} > 0$ such that the open ball of radius $\varepsilon_{i}$ around $\mathbf{c}_{i}$ is also contained in the cone. Let $\varepsilon = \min(\varepsilon_{i})$, and consider the balls $D_{\mathbf{w}}$ and $D_{\mathbf{c}}$ of radius $\varepsilon$ around $\mathbf{w}$ and $\mathbf{c}$. Then any line segment from a point in $D_{\mathbf{w}}$ to a point in $D_{\mathbf{c}}$ must intersect each ball of radius $\varepsilon$ around each $\mathbf{c}_{i}$ and therefore must intersect each cone $C_{i}$ on its interior. 

 By applying Lemma \ref{lem:improvedsubdiv} twice, we may construct a polytope $Q$ with $m + 2n$ facets such that the normal fan of $Q$ is the same as the normal fan of $P$ except the normal cones at $\mathbf{u}$ and $\mathbf{v}$ are subdivided. In particular, we may construct $Q$ such that there are vertices $\mathbf{a}, \mathbf{b}$ of $Q$ with normal cones $C_{\mathbf{a}}$ and $C_{\mathbf{b}}$ such that the extreme rays of the cones are contained in $D_{\mathbf{w}}$ and $D_{\mathbf{c}}$ respectively. Consider a shadow path from $\mathbf{a}$ to $\mathbf{b}$. Then there must be a vector $\mathbf{w}_{\mathbf{a}} \in C_{\mathbf{a}}$ and $\mathbf{c}_{\mathbf{b}} \in C_{\mathbf{b}}$ such that the vertices of the path are given by precisely the set of vertices corresponding to cones intersecting the line segment from $\mathbf{w}_{\mathbf{a}}$ to $\mathbf{c}_{\mathbf{b}}$. Since the rays of $C_{\mathbf{a}}$ are contained in $D_{\mathbf{w}}$, we may rescale $\mathbf{w}_{\mathbf{a}}$ to find a vector in $D_{\mathbf{w}}$. Rescaling does not change the shadow path, since the shadow path is dictated by the ordering of the slopes of the edges of the polytope under the shadow map $\pi$, and rescaling changes the slopes of all of the edges by the same positive constant factor. Hence, without loss of generality, we may assume that $\mathbf{w}_{\mathbf{a}} \in C_{\mathbf{a}} \cap D_{\mathbf{w}}$ and, by similar reasoning, that $\mathbf{c}_{\mathbf{b}} \in C_{\mathbf{b}} \cap D_{\mathbf{c}}$. It follows that the line segment from $\mathbf{w}_{\mathbf{a}}$ to $\mathbf{c}_{\mathbf{b}}$ intersects each of the $C_{i}$ on their interior. Since each $C_{i}$ other than $C_{0}$ and $C_{\alpha}$ are contained in the normal fan of $Q$, the number of normal cones intersected to arrive at the shadow path in $Q$ is at least $\alpha$. Therefore, since our choice of $\mathbf{w}_{\mathbf{a}}$ and $\mathbf{c}_{\mathbf{b}}$ was arbitrary, the length of any shadow path from $\mathbf{a}$ to $\mathbf{b}$ is at least $\alpha$.
\end{proof}

The proof of Theorem \ref{thm:coherentpath} then follows immediately from applying Theorem \ref{thm:manyfromone} to a combinatorial cube with an shadow containing all of its vertices as in \cite{GoldfarbCube}. Note that, if $\mathbf{c}$ is fixed, only one cone needs to be subdivided. If the linear program is fixed, we may assume that $\mathbf{c}$ is fixed, so the result could be sharpened to allow for $m+n$ facets in $n$ dimensions, which implies Theorem \ref{thm:shadowsimplex}. One challenge of this result is that it is not constructive as is currently presented. In particular, the coefficients of inequalities defining the resulting polytope may blow up in size. A lower bound that ensures the bit-size of the description of the polytope is not too large may still be possible with our techniques.

From the theory of monotone path polytopes \cite{FiberPoly}, given a fixed choice of $\mathbf{c} \in \mathbb{R}^{n}$, the set of $\mathbf{w}$ that yield the same coherent monotone path is a polyhedral cone. Namely, a coherent monotone path is chosen from an initial vertex precisely when each step is the path has maximal slope amongst all improving neighbors. Since $\mathbf{c}$ is fixed, all inequalities between the slopes of edges of the polytope for different choices of $\mathbf{w}$ are linear inequalities. Hence, the set of $\mathbf{w}$ choosing the same coherent monotone path is defined by linear inequalities in terms of the edge directions of the polytope and the linear objective function $\mathbf{c}$. Now, instead of using the nonconstructive balls of radius $\varepsilon$, we may work with an explicit polyhedral cone of all choices of $\mathbf{w}$ that yield the same coherent monotone path. Then the same construction technique applies, but instead of building a cone with rays in the ball of radius $\varepsilon$, we build a new normal cone inside the cone of all $\mathbf{w}$ defining the coherent monotone path, which has an explicit inequality description in terms of the polytope. The critical challenge here is that the cone for the coherent monotone path may be defined by exponentially many inequalities, so there is no guarantee of the existence of a cone contained within it defined by polynomial data.

\begin{op}
Does there exist a shadow pivot rule that guarantees the simplex method will run in weakly polynomial run-time? 
\end{op}

Furthermore, our construction method requires a minimum of $3n$ facets in $n$ dimensions to yield an exponential counterexample for every shadow pivot rule. This is distinct from many other counterexamples for other pivot rules, which are combinatorial cubes. The following open question remains interesting to push this result further:

\begin{op}
For each $n \in \mathbb{N}$, does there exist an $n$-dimensional combinatorial cube $C_{n}$ for which every shadow path between a fixed pair of vertices of $C_{n}$ is of length exponential in $n$? 
\end{op}

Even outside of the context of the simplex method, it is surprising that one can construct a polytope for which any shadow containing a certain fixed pair of vertices has exponentially many vertices. The following open problem is a follow up to that observation:

\begin{op}
Does there exist a family of polytopes for which each shadow containing a certain fixed vertex has exponentially many vertices?
\end{op}

 This problem is equivalent to asking for an example where every shadow path from a $\mathbf{c}$-minimum to a $\mathbf{c}$-maximum is of exponential length. A final related open problem is a strengthening of this:

\begin{op}
\label{op:everyshadowbad}
Does there exist a family of polytopes for which every shadow has exponentially many vertices?
\end{op}

By applying a suitable invertible linear transformation, one can ensure that any particular shadow of the polytope appears with arbitrarily high probability for a random shadow. Hence, Open Problem \ref{op:everyshadowbad} is equivalent to asking whether there exists a family polytopes closed under invertible linear transformations for which the expected size of a random shadow is always exponential.

\section{Steepest Edge Lower Bounds for Every Norm}

In this section, our goal is to prove Theorems \ref{thm:anynorm} and \ref{thm:allnormsatonce}. The key tool for our construction method is the family of exponential examples for the shadow-simplex method by Goldfarb in \cite{GoldfarbCube}. Namely, Goldfarb constructed a family of combinatorial cubes for which the shadow simplex method follows a path through all the vertices on those cubes. For any fixed norm, we construct an affine transformation of a Golfarb cube that guarantees the simplex method with the steepest edge variant corresponding to the chosen norm will follow the same path.  See Figure \ref{fig:Compressing} for an illustration of our approach.

\begin{proof}[Proofs of Theorems \ref{thm:anynorm} and \ref{thm:allnormsatonce}]
Let $C_{n}$ be an $n$-dimensional combinatorial cube for which a shadow simplex path walks through all the vertices for optimizing a linear program $\max(\mathbf{c}^{\intercal} \mathbf{x})$ such that $\mathbf{x} \in C_{n}$. Let $\mathbf{w}$ be any auxiliary vector that chooses the shadow simplex path walking through all vertices.

\begin{figure}
\[    \begin{tikzpicture}[xscale = .3, yscale = .3]
    \draw (6,12) node {$C_{3} = A_{1}(C_{3})$};
    \draw[thick, red] (0,-.5) -- (4,0) -- (8,1) -- (12,4) -- (12,6) -- (8,9) -- (4,10) -- (0,10.5);
    \draw[thick] (0,-.5) -- (0,10.5);
    \draw[thick, dashed] (0,-.5) -- (12,4);
    \draw[thick, dashed] (0,10.5) -- (12,6);
    \draw[thick] (4,0) -- (4,10);
    \draw[thick] (8,1) -- (8,9);
    \draw (0,10.5) node[orange, circle,fill, inner sep = 1.5pt] {};
\end{tikzpicture} \hspace{.5 cm}
\begin{tikzpicture}[xscale = .2, yscale = .3]
    \draw (6,12) node {$A_{2}(C_{3})$};
    \draw[thick, red] (0,-.5) -- (4,0) -- (8,1) -- (12,4) -- (12,6) -- (8,9) -- (4,10) -- (0,10.5);
    \draw[thick] (0,-.5) -- (0,10.5);
    \draw[thick, dashed] (0,-.5) -- (12,4);
    \draw[thick, dashed] (0,10.5) -- (12,6);
    \draw[thick] (4,0) -- (4,10);
    \draw[thick] (8,1) -- (8,9);
    \draw (0,10.5) node[orange, circle,fill, inner sep = 1.5pt] {};
\end{tikzpicture}\hspace{.5 cm}
\begin{tikzpicture}[xscale = .1, yscale = .3]
    \draw (6,12) node {$A_{3}(C_{3})$};
    \draw[thick, red] (0,-.5) -- (4,0) -- (8,1) -- (12,4) -- (12,6) -- (8,9) -- (4,10) -- (0,10.5);
    \draw[thick] (0,-.5) -- (0,10.5);
    \draw[thick, dashed] (0,-.5) -- (12,4);
    \draw[thick, dashed] (0,10.5) -- (12,6);
    \draw[thick] (4,0) -- (4,10);
    \draw[thick] (8,1) -- (8,9);
    \draw (0,10.5) node[orange, circle,fill, inner sep = 1.5pt] {};
\end{tikzpicture}\hspace{.5 cm}
\begin{tikzpicture}[xscale = .05, yscale = .3]
    \draw (6,12) node {$A_{4}(C_{3})$};
    \draw[thick, red] (0,-.5) -- (4,0) -- (8,1) -- (12,4) -- (12,6) -- (8,9) -- (4,10) -- (0,10.5);
    \draw[thick] (0,-.5) -- (0,10.5);
    \draw[thick, dashed] (0,-.5) -- (12,4);
    \draw[thick, dashed] (0,10.5) -- (12,6);
    \draw[thick] (4,0) -- (4,10);
    \draw[thick] (8,1) -- (8,9);
    \draw (0,10.5) node[orange, circle,fill, inner sep = 1.5pt] {};
\end{tikzpicture}\hspace{.5cm}
\begin{tikzpicture}[xscale = .025, yscale = .3]
    \draw (6,12) node {$A_{5}(C_{3})$};
    \draw[thick, red] (0,-.5) -- (4,0) -- (8,1) -- (12,4) -- (12,6) -- (8,9) -- (4,10) -- (0,10.5);
    \draw[thick] (0,-.5) -- (0,10.5);
    \draw[thick, dashed] (0,-.5) -- (12,4);
    \draw[thick, dashed] (0,10.5) -- (12,6);
    \draw[thick] (4,0) -- (4,10);
    \draw[thick] (8,1) -- (8,9);
    \draw (0,10.5) node[orange, circle,fill, inner sep = 1.5pt] {};
\end{tikzpicture}\hspace{.5cm}
\begin{tikzpicture}[xscale = .01, yscale = .3]
    \draw (6,12) node {$A_{6}(C_{3})$};
    \draw[thick, red] (0,-.5) -- (4,0) -- (8,1) -- (12,4) -- (12,6) -- (8,9) -- (4,10) -- (0,10.5);
    \draw[thick] (0,-.5) -- (0,10.5);
    \draw[thick, dashed] (0,-.5) -- (12,4);
    \draw[thick, dashed] (0,10.5) -- (12,6);
    \draw[thick] (4,0) -- (4,10);
    \draw[thick] (8,1) -- (8,9);
    \draw (0,10.5) node[orange, circle,fill, inner sep = 1.5pt] {};
\end{tikzpicture}
\]
    \caption{Pictured is a combinatorial cube $C_{3}$ with a shadow containing all of its vertices as is used in the proof of Theorems \ref{thm:anynorm} and \ref{thm:allnormsatonce}. In the proof, we rely on a family of linear transformations $A_{k}$ that make the edges of the polytope approximately parallel to ensure a steepest edge pivot rule will follow the shadow path. The picture illustrates $A_{k}(C_{3})$ as k increases. }
    \label{fig:Compressing}
\end{figure}
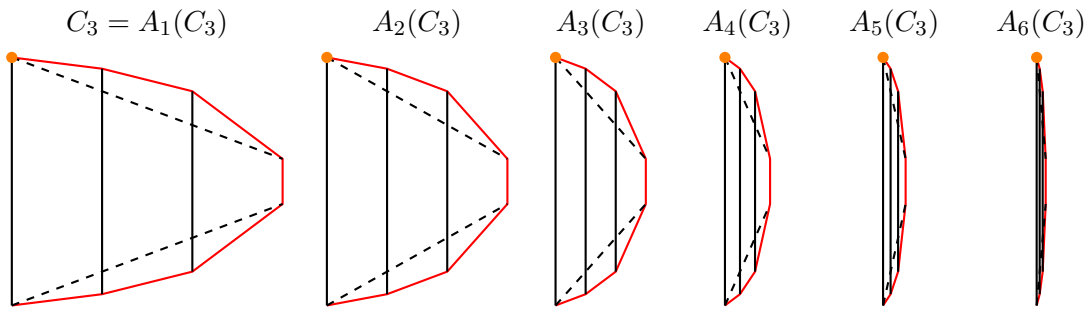
We may assume without loss of generality that $\|\mathbf{w}\|_{2} = 1$, since rescaling does not change the choice of shadow. Construct an orthonormal basis for $\mathbb{R}^{n}$ with respect to the typical $2$-norm via the Gram Schmidt process given by $\mathbf{w}, \mathbf{a}^{1}, \dots, \mathbf{a}^{n-1}$. Then define a family of linear maps for all $k > 0$ 
\[A_{k}\left(\lambda \mathbf{w} + \sum_{i=1}^{n-1} \lambda_{i} \mathbf{a}^{i}\right) = \lambda \mathbf{w} + \sum_{i=1}^{n-1} \frac{ \lambda_{i}}{k} \mathbf{a}^{i}. \]

By construction, $A_{k}$ is invertible for all $k > 0$ and $A_{k}(\mathbf{w}) = \mathbf{w}$. Our goal is to show that for all $k$ sufficiently large and vertices $\mathbf{v}$ of $C_{n}$, we have
\begin{equation}
\label{eqn:initialformulation}
    \argmax_{\mathbf{u} \in N_{\mathbf{c}}(\mathbf{v})}  \left( \frac{\mathbf{c}^{\intercal} A_{k}^{-1}(A_{k}(\mathbf{u}) - A_{k}(\mathbf{v}))}{\mathbf{w}^{\intercal}A_{k}^{-1}(A_{k}(\mathbf{u}) - A_{k}(\mathbf{v}))}\right) = \argmax_{\mathbf{u} \in N_{\mathbf{c}}(\mathbf{v})}  \left( \frac{\mathbf{c}^{\intercal}A_{k}^{-1}(A_{k}\mathbf{u} - A_{k}\mathbf{v})}{\eta(A_{k}\mathbf{u} - A_{k}\mathbf{v})}\right),    
\end{equation}

for that readily implies that the $\eta$-steepest edge path on $A_{k}(P)$ must be exponentially long. Observe that Equation \ref{eqn:initialformulation} may be rewritten as
\[\argmax_{\mathbf{u} \in N_{\mathbf{c}}(\mathbf{v})}  \left( \frac{\mathbf{c}^{\intercal}(\mathbf{u} - \mathbf{v})}{\mathbf{w}^{\intercal}(\mathbf{u} - \mathbf{v})}\right) = \argmax_{\mathbf{u} \in N_{\mathbf{c}}(\mathbf{v})}  \left( \frac{\mathbf{c}^{\intercal}(\mathbf{u} - \mathbf{v})}{\eta(A_{k}(\mathbf{u} -\mathbf{v}))}\right).\]
In fact, rescaling $\mathbf{w}$ does not change the chosen vertex. Thus, the following is also equivalent to Equation \ref{eqn:initialformulation}: 
\[\argmax_{\mathbf{u} \in N_{\mathbf{c}}(\mathbf{v})}  \left( \frac{\mathbf{c}^{\intercal}(\mathbf{u} - \mathbf{v})}{\eta(\mathbf{w})\mathbf{w}^{\intercal}(\mathbf{u} - \mathbf{v})}\right) = \argmax_{\mathbf{u} \in N_{\mathbf{c}}(\mathbf{v})}  \left( \frac{\mathbf{c}^{\intercal}(\mathbf{u} - \mathbf{v})}{\eta(A_{k}(\mathbf{u} -\mathbf{v}))}\right).\]

Note that $\mathbf{c}$, $\mathbf{w}$, and $C_{n}$ may be perturbed slightly to be sufficiently generic so that $\frac{\mathbf{c}^{\intercal}(\mathbf{u}-\mathbf{v)}}{\mathbf{w}^{\intercal}(\mathbf{u}-\mathbf{v})}$ is distinct for all pairs of vertices $\mathbf{u}$ and $\mathbf{v}$ of $P$ such that $\mathbf{u}$ is a $\mathbf{c}$-improving neighbor of $\mathbf{v}$. In particular, there exists $\varepsilon > 0$ sufficiently small such that if 
\[|\eta(A_{k}(\mathbf{u}-\mathbf{v})) - \eta(\mathbf{w})\mathbf{w}^{\intercal}(\mathbf{u}-\mathbf{v})| < \varepsilon, \]
for all $\mathbf{u}, \mathbf{v} \in P$ with $\mathbf{u}$ a $\mathbf{c}$-improving neighbor of $\mathbf{v}$, then Equation \ref{eqn:initialformulation} must hold.

Let $S = \{\mathbf{u} - \mathbf{v}: \mathbf{u} \in N_{\mathbf{c}}(\mathbf{v}), \mathbf{v} \in V(C_{n})\}$, where $V(C_{n})$ is the set of vertices of $C_{n}$. To prove that Equation \ref{eqn:initialformulation} holds, it suffices to show there exists $k$ sufficiently large that 
\[\max_{\mathbf{s} \in S} (|\eta(A_{k}(\mathbf{s})) - \eta(\mathbf{w})\mathbf{w}^{\intercal}\mathbf{s}|) < \varepsilon.\]
In fact, it suffices to show for any $\mathbf{s} \in S$ one can find $k_{\mathbf{s}}$ sufficiently large so that 
\[|\eta(A_{k}(\mathbf{s})) - \eta(\mathbf{w})\mathbf{w}^{\intercal} \mathbf{s}| < \varepsilon\]
for all $k > k_{\mathbf{s}}$, since one can take the maximum of all choices of $k_{\mathbf{s}}$ to achieve the uniform bound. This will be our strategy.

Let $\mathbf{t} \in S$. We now setup the technical bounds for our computation. The first steps are purely algebraic taking advantage of $\eta$ being positive homogeneous:

\begin{align*}
    |\eta(A_{k}(\mathbf{t})) - \eta(\mathbf{w})\mathbf{w}^{\intercal} \mathbf{t}| &= |\eta(A_{k}((\mathbf{w}^{\intercal} \mathbf{t}) \mathbf{w} + \mathbf{t} - (\mathbf{w}^{\intercal} \mathbf{t}) \mathbf{w})) - \eta(\mathbf{w}) \mathbf{w}^{\intercal} \mathbf{t}|, \\
    &= \left||\mathbf{w}^{\intercal} \mathbf{t}| \eta\left(\mathbf{w} +A_{k}\left( \frac{\mathbf{t} - \mathbf{w}^{\intercal} \mathbf{t} \mathbf{w}} {\mathbf{w}^{\intercal} \mathbf{t}}\right)\right)  - \eta(\mathbf{w}) \mathbf{w}^{\intercal} \mathbf{t}\right|. \\
\end{align*}

Next we rely on an observation coming from the structure of our example $C_{n}$. By Lemma $2$ of \cite{01simplex}, the shadow simplex path is both strictly $\mathbf{c}$-increasing and $\mathbf{w}$-increasing. Since the path on $C_{n}$ contains all the vertices of $P$, $\mathbf{c}$ and $\mathbf{w}$ must induce the same ordering on the vertices of $P$. In particular, an edge is $\mathbf{c}$-improving if and only if it is $\mathbf{w}$-improving.   It follows that $\mathbf{w}^{\intercal}(\mathbf{u} - \mathbf{v}) > 0$ for all $\mathbf{u}$ and $\mathbf{v}$ such that $\mathbf{u}$ is a $\mathbf{c}$-improving neighbor of $\mathbf{v}$ meaning that $\mathbf{w}^{\intercal} \mathbf{s} > 0$ for all $\mathbf{s} \in S$. This observation ensures that the sign of $\mathbf{w}^{\intercal} \mathbf{t}$ and $\eta(\mathbf{t})$ coincide. In paritcular, this allows us to factor out $\mathbf{w}^{\intercal} \mathbf{t}$ to find  
\[|\eta(A_{k}(\mathbf{t})) - \eta(\mathbf{w})\mathbf{w}^{\intercal} \mathbf{t}| = \mathbf{w}^{\intercal} \mathbf{t} \left|\eta\left(\mathbf{w} + A_{k}\left(\frac{\mathbf{t} - \mathbf{w}^{\intercal} \mathbf{t} \mathbf{w}} {\mathbf{w}^{\intercal} \mathbf{t}}\right)\right) - \eta(\mathbf{w}) \right|.\]

Recall that norms are continuous, so there exists $\delta > 0$ sufficiently small and $\mathbf{y} \in \mathbb{R}^{n}$, such that whenever $\|\mathbf{y}\|_{1} < \delta$,
\[|\eta(\mathbf{w} + \mathbf{y}) - \eta(\mathbf{w})| < \frac{\varepsilon}{\mathbf{w}^{\intercal} \mathbf{t}}. \]

Furthermore, there exists $k_{\mathbf{t}}$ so large that that for all $k > k_{\mathbf{t}}$, by the definition of $A_{k}$,
\[\left\|A_{k}\left(\frac{\mathbf{t} - (\mathbf{w}^{\intercal} \mathbf{t} \mathbf{w})}{\mathbf{w}^{\intercal} \mathbf{t}}\right) \right\|_{1} =  \left\|\frac{1}{k\mathbf{w}^{\intercal} \mathbf{t}}(\mathbf{t} - \mathbf{w}^{\intercal} \mathbf{t} \mathbf{w})\right\|_{1} < \delta. \]
Hence, for all $k > k_{\mathbf{t}}$, 
\[|\eta(A_{k}(\mathbf{t})) - \eta(\mathbf{w})\mathbf{w}^{\intercal} \mathbf{t}| = \mathbf{w}^{\intercal} \mathbf{t} \left|\eta\left(\mathbf{w} + A_{k}\left(\frac{\mathbf{t} - \mathbf{w}^{\intercal} \mathbf{t} \mathbf{w}} {\mathbf{w}^{\intercal} \mathbf{t}}\right)\right) - \eta(\mathbf{w}) \right| < (\mathbf{w}^{\intercal} \mathbf{t}) \frac{\varepsilon}{\mathbf{w}^{\intercal} \mathbf{t}} = \varepsilon,\]
which completes the proof of Theorem \ref{thm:anynorm}.

To prove Theorem \ref{thm:allnormsatonce}, note that there are two choices we make that depend on the norm. The first is the value of $\varepsilon$, which depends on the value of $\eta(\mathbf{w})$. Without loss of generality, up to an invertible linear transformation, we may assume that $\mathbf{w} = e_{1}$. Then, since $\eta$ is regular, $\eta(\mathbf{w}) = 1$. Hence, the choice of $\varepsilon$ will be the same for all regular norms.

The choice of $k_{\mathbf{t}}$ depended solely on the choice of $\delta$ independent of the norm. Namely, we need $\delta$ to be made so small that whenever $\|\mathbf{y}\|_{1} < \delta$,
\[|\eta(\mathbf{w} + \mathbf{y}) - \eta(\mathbf{w})| < \frac{\varepsilon}{\mathbf{w}^{\intercal} \mathbf{t}}. \]
Thus, to prove Theorem \ref{thm:allnormsatonce}, it suffices to show that a singular choice of $\delta$ works for all regular norms. Suppose first that $\eta(\mathbf{w} + \mathbf{y}) \geq \eta(\mathbf{w})$. Then, by the triangle inequality,
\[\eta(\mathbf{w} + \mathbf{y}) - \eta(\mathbf{w}) \leq \eta(\mathbf{w}) + \eta(\mathbf{y}) - \eta(\mathbf{w}) = \eta(\mathbf{y}). \]
Similarly, if $\eta(\mathbf{w} + \mathbf{y}) < \eta(\mathbf{w})$, by the triangle inequality
\[\eta(\mathbf{w}) - \eta(\mathbf{w} + \mathbf{y}) = \eta(\mathbf{w} + \mathbf{y} - \mathbf{y}) - \eta(\mathbf{w} + \mathbf{y}) \leq \eta(-\mathbf{y}) + \eta(\mathbf{w}+ \mathbf{y}) - \eta(\mathbf{w} + \mathbf{y}) = \eta(\mathbf{y}).  \]
Hence, $|\eta(\mathbf{w} + \mathbf{y}) - \eta(\mathbf{w})| \leq \eta(\mathbf{y})$. Then, for any regular norm, whenever $\|\mathbf{y}\|_{1} < \frac{\varepsilon}{\mathbf{w}^{\intercal} \mathbf{x}}$, by the triangle inequality
\[|\eta(\mathbf{w} + \mathbf{y}) - \eta(\mathbf{w})| \leq \eta(\mathbf{y}) \leq \sum_{i=1}^{n} \eta(\mathbf{y}_{i} e_{i}) = \sum_{i=1}^{n} |y_{i}| \eta(e_{i}) = \sum_{i=1}^{n} |y_{i}| = \|\mathbf{y}\|_{1} < \frac{\varepsilon}{\mathbf{w}^{\intercal} \mathbf{x}}.\]
Hence, $\delta = \frac{\varepsilon}{\mathbf{w}^{\intercal} \mathbf{x}}$ works simultaneously for any regular norm meaning that there exists a combinatorial cube for which the path through all the vertices is chosen for the steepest edge for every regular norm simultaneously. 
\end{proof}

While we state our results for norms, our proof technique is stronger, and one can generalize Theorem \ref{thm:anynorm} to a broader class of functions. Namely, so long as the normalization function is fixed, the same argument yields exponential lower bounds for any normalization function that is continuous, positive homogeneous  (i.e., $f(\lambda \mathbf{x}) = |\lambda| f(\mathbf{x})$), and such that $f(\mathbf{x}) = 0$ if and only if $\mathbf{x} = \mathbf{0}$. In particular, we did not require the triangle inequality until we searched for a uniform bound independent of the norm.  This is a large class of functions encompassing every continuous function from the $n$-sphere to the positive real line and extended to all of $\mathbb{R}^{n}$ to force positive homogeneity. 

One could ask for an example that fails for all positively homogeneous continuous functions simultaneously, which would vastly generalize the result of Theorem \ref{thm:allnormsatonce}. However, by the argument in the proof of Theorem $1.3$ of \cite{pivotpolytopes}, finding a generic linear program where every positively homogeneous, continuous normalization function yields an exponential length path is equivalent to the monotone variant of the polynomial Hirsch conjecture. Thus, finding such an example is likely very hard or impossible. Doing so for norms is more likely to be feasible:

\begin{op}
Does there exist a family of linear programs for which the steepest edge path for any norm is always of exponential length?
\end{op}

A particularly interesting special case not covered by our results here is for norms of the form $\eta(\mathbf{x}) = \|A \mathbf{x}\|_{2}$ for an invertible linear transformation $A$. An exponential lower bound for all norms of that type would imply the existence of a family of polytopes for which the $2$-norm steepest edge pivot rule would follow an exponential length path for some linear program on any affine transformation of those polytopes. It still remains interesting to see if one can weaken our assumptions for our results further in the non-uniform case as well. We conclude with the following open problem:

\begin{op}
\label{op:steepness}
Consider the set of functions $\mathcal{C}$ such that $f \in \mathcal{C}$ if $f: \mathbb{R}^{n} \to \mathbb{R}$ is continuous and $f(\mathbf{x}) = 0$ if and only if $\mathbf{x} = \mathbf{0}$. Is it true that for all $f \in \mathcal{C}$, there exists a family of polytopes for which the following pivot rule requires exponentially many steps to solve an LP:
\[\mathbf{v}^{i+1} = \argmax_{\mathbf{u} \in N_{\mathbf{c}}(\mathbf{v}^{i})} \frac{\mathbf{c}^{\intercal}(\mathbf{u} - \mathbf{v}^{i})}{f(\mathbf{u} -\mathbf{v}^{i})}, \]
where $N_{\mathbf{c}}(\mathbf{v}^{i})$ is the set of $\mathbf{c}$-improving neighbors of $\mathbf{v}^{i}$?
\end{op}

Note that the argument of Theorem \ref{thm:anynorm} does not directly extend to this problem. We can use the same observation and apply a transformation in which all edges are almost parallel with a fixed direction $\mathbf{w}$. By continuity, $f(\mathbf{x}) \sim f((\mathbf{w}^{\intercal} \mathbf{x}) \mathbf{w})$. If $f$ is positively homogenous, the function $g(\lambda) := f(\lambda \mathbf{w})$ satisfies $g(\lambda) = |\lambda| f(\mathbf{w})$, and that allows for our proof. However, in general, $g$ could be any continuous real-valued function, which forces the argument to break down. An intuitive explanation of our results is that there is no uniform method to find a polynomial length simplex path purely in terms of the edge directions of the polytope. If the answer to open Problem \ref{op:steepness} is yes, then even knowing the edges and their lengths would be insufficient without taking into account more information about the polytope.

\section*{Acknowledgments}
The author was supported by the NSF Graduate Research Fellowship Program (GRFP) and appreciates comments of Jes\'{u}s De Loera and Sean Kafer on earlier versions of this manuscript. The author would also like to thank Sophie Huiberts for noting a connection between the $2$-norm steepest edge rule and shadow pivot rules on affine transformations of polytopes.

\bibliographystyle{amsplain}
\bibliography{bibliography.bib} 

	
\end{document}